\documentclass[12pt]{article}

\usepackage{amssymb}
\usepackage{url}
\usepackage{amscd}
\usepackage{amsthm}
\usepackage{amsmath}
\usepackage{graphicx,tikz,float}
\usepackage[figuresleft]{rotating}

\newtheorem {Theorem}  {Theorem}
\newtheorem {Corollary} {Corollary}

\newtheorem {Definition} {Definition}
\newtheorem {Fact} {Fact}

\begin{document}
\baselineskip = 15pt
\bibliographystyle{plain}

\title{Undecidability of tiling the plane with a fixed number of Wang bars}
\date{}
\author{Chao Yang\\ 
              School of Mathematics and Statistics\\
              Guangdong University of Foreign Studies, Guangzhou, 510006, China\\
              sokoban2007@163.com, yangchao@gdufs.edu.cn\\
              \\
        Zhujun Zhang\\
              Government Data Management Center of\\
              Fengxian District, Shanghai, 201499, China\\
              zhangzhujun1988@163.com\\
              }

\maketitle

\begin{abstract}
To study the fixed parameter undecidability of tiling problem for a set of Wang tiles, Jeandel and Rolin show that the tiling problem for a set of $44$ Wang bars is undecidable. In this paper, we improve their result by proving that whether a set of $29$ Wang bars can tile the plane is undecidable. As a consequence, the tiling problem for a set of Wang tiles with color deficiency of $25$ is also undecidable.
\end{abstract}

\noindent{\textbf{Keywords}}:
Wang tiles, Wang bars, plane tiling, undecidability, color deficiency\\
MSC2020:  52C20, 68Q17

\section{Introduction} 

Wang's domino problem \cite{wang61} is one of the earliest known and most representative undecidable problems in tiling \cite{p14}. A \textit{Wang tile} is a unit square whose edges are assigned colors. See Figure \ref{fig_wang_set} for an example of a set of Wang tiles. For the purpose we will mention later, the edges of Wang tiles in Figure \ref{fig_wang_set} are parallel to lines of slopes $1$ or $-1$.


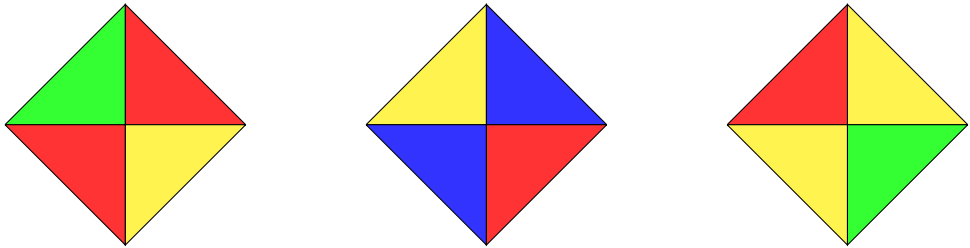
\begin{figure}[H]
\begin{center}
\begin{tikzpicture}[scale=0.8]

\draw [ fill=green!80] (0,0)--(2,2)--(2,0)--(0,0);
\draw [ fill=red!80] (0,0)--(2,-2)--(2,0)--(0,0);
\draw [ fill=red!80] (4,0)--(2,2)--(2,0)--(4,0);
\draw [ fill=yellow!80] (4,0)--(2,-2)--(2,0)--(4,0);

\draw [ fill=yellow!80] (6+0,0)--(6+2,2)--(6+2,0)--(6+0,0);
\draw [ fill=blue!80] (6+0,0)--(6+2,-2)--(6+2,0)--(6+0,0);
\draw [ fill=blue!80] (6+4,0)--(6+2,2)--(6+2,0)--(6+4,0);
\draw [ fill=red!80] (6+4,0)--(6+2,-2)--(6+2,0)--(6+4,0);

\draw [ fill=red!80] (12+0,0)--(12+2,2)--(12+2,0)--(12+0,0);
\draw [ fill=yellow!80] (12+0,0)--(12+2,-2)--(12+2,0)--(12+0,0);
\draw [ fill=yellow!80] (12+4,0)--(12+2,2)--(12+2,0)--(12+4,0);
\draw [ fill=green!80] (12+4,0)--(12+2,-2)--(12+2,0)--(12+4,0);

\end{tikzpicture}
\end{center}
\caption{A set of $3$ Wang tiles.}\label{fig_wang_set}
\end{figure}

\begin{Definition}[Wang's domino problem]
    Given a set of Wang tiles, is it possible to tile the entire plane with translated copies from the set, such that the tiles are edge-to-edge, and the common edge of any two adjacent tiles must have the same color?
\end{Definition}

Berger \cite{b66} proved the following result.

\begin{Theorem}[\cite{b66}]\label{thm_berger}
     Wang's domino problem is undecidable.
\end{Theorem}

Berger's original proof requires an explicit construction of a set of over $20000$ Wang tiles. Several simpler proofs \cite{k08,o08,r71} are found later. Besides the interest in looking for simpler proofs, there is also quite a lot of research on the undecidability of the generalizations and variants of Wang's domino problem. The finite tiling problem \cite{k94}, periodic domino tiling problem \cite{j10}, infinite snake tiling problem \cite{k03}, domino problem on rhombus subshift \cite{hln23}, rectangular tiling problem, and complementary tiling problem \cite{y14} were all shown to be undecidable. Undecidable tiling problems on the hyperbolic plane are studied in \cite{k08,m08,m23,r78}. 

Ollinger initiated the study of the tiling problem of the plane by translated copies of a fixed number of polyominoes \cite{o09}, which is called the \textit{$k$-polyomino tiling problem} for a fixed integer $k$. Ollinger showed that $11$-polyomino tiling problem is undecidable. Ollinger's result has been improved by Yang and Zhang in \cite{yang23,yang24,yz24}.

Jeandel and Rolin tried to obtain fixed parameter undecidability for the original Wang tiles \cite{jr12}. But fixing the number of tiles or colors of a set of Wang tiles does not work, because the problem has only a finite number of instances in these cases. So they defined the following parameter for sets of Wang tiles, which we will call \textit{color deficiency}.

\begin{Definition}[Color deficiency]
    Given a set $T$ of Wang tiles, let $n(T)=|T|$ be the number of tiles in the set, and let $c(T)$ be the maximum number of different colors among the four sides in the set. The color deficiency, denoted by $cd(T)$ is defined by $cd(T)=n(T)-c(T).$
\end{Definition}

It is obvious from the definition that the color deficiency must be non-negative. As an example, the color deficiency of the set of Wang tiles in Figure \ref{fig_wang_set} is $0$. To obtain undecidability result for Wang's domino problem with a fixed color deficiency, Jeandel and Rolin also introduced the plane tiling problem of a set of Wang bars. A \textit{Wang bar} is an extension of Wang tile such that the length of its horizontal sides can be any positive integer besides $1$, and each unit segment of the horizontal or vertical sides is assigned a color. Figure \ref{fig_wang_bar} illustrates a Wang bar of length $3$. A Wang tile is a Wang bar of length $1$.


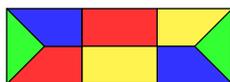
\begin{figure}[H]
\begin{center}
\begin{tikzpicture}

\draw (0,0)--(3,0)--(3,1)--(0,1)--(0,0);

\draw [fill=red!80] (0,0)--(0.5,0.5)--(1,0.5)--(1,0)--(0,0);
\draw [fill=blue!80] (0,1)--(0.5,0.5)--(1,0.5)--(1,1)--(0,1);

\draw [fill=red!80] (1,1)--(2,1)--(2,0.5)--(1,0.5)--(1,1);
\draw [fill=yellow!80] (1,0)--(2,0)--(2,0.5)--(1,0.5)--(1,0);

\draw [fill=yellow!80] (2,1)--(3,1)--(2.5,0.5)--(2,0.5)--(2,1);
\draw [fill=blue!80] (2,0)--(3,0)--(2.5,0.5)--(2,0.5)--(2,0);

\draw [fill=green!80] (0,0)--(0,1)--(0.5,0.5)--(0,0);
\draw [fill=green!80] (3,0)--(3,1)--(2.5,0.5)--(3,0);

\foreach \x in {1,...,2}
{
\draw (\x,1)--(\x,0);
}
\end{tikzpicture}
\end{center}
\caption{A Wang bar.}\label{fig_wang_bar}
\end{figure}

Jeandel and Rolin considered the following tiling problem for Wang bars.

\begin{Definition}[Tiling problem for $k$ Wang bars]
    Let $k$ be a fixed positive integer. Given a set of $k$ Wang bars, is it possible to tile the entire plane with translated copies of Wang bars from the set, such that the color of each unit segment matches (i.e. with the same color) for any two adjacent Wang bars in the tiling?
\end{Definition}

The following undecidability result is obtained in \cite{jr12}, by applying the idea of Ollinger in proving the undecidability of tiling with a fixed number of polyominoes \cite{o09}.

\begin{Theorem}[\cite{jr12}]\label{thm_wang_bar_44}
    The tiling problem for $44$ Wang bars is undecidable.
\end{Theorem}

Jeandel and Rolin also claimed without proof that the number of Wang bars in the above result can be optimized to $35$. Moreover, they found a relation between the tiling problem for a fixed number of Wang bars and the tiling problem for Wang tiles with a fixed color deficiency.

\begin{Theorem}[\cite{jr12}]\label{thm_cd}
    If the tiling problem for $k$ Wang bars is undecidable, then the tiling problem for Wang tiles with color deficiency $k-1$ is undecidable.
\end{Theorem}

In this paper, we improve Theorem \ref{thm_wang_bar_44} considerably by proving Theorem \ref{thm_main}. Our reduction method is significantly different from that of Jeandel and Rolin \cite{jr12}, though still following Ollinger's method in spirit \cite{o09}. As a consequence, our result also implies the undecidability of tiling problem for Wang tiles with a smaller fixed color deficiency.

\begin{Theorem}\label{thm_main}
    The tiling problem for $29$ Wang bars is undecidable.
\end{Theorem}

The rest of the paper is organized as follows. Section \ref{sec_main} gives the proof of Theorem \ref{thm_main}. Section \ref{sec_con} concludes with a few remarks.

\section{Proof of Theorem \ref{thm_main}}\label{sec_main}

\subsection{Overview of the Proof}

We will prove Theorem \ref{thm_main} by reduction from Wang's domino problem in this section. Ollinger first introduced a framework for reducing Wang's domino problem to tiling problems with a fixed number of tiles. One key ingredient in Ollinger's framework is a tile called the \textit{meat} polyomino. The meat encodes a set of an arbitrary number of Wang tiles into a single polyomino. Another important ingredient in Ollinger's framework is the mechanism to select an encoded Wang tile from every meat polyomino. The \textit{jaw} polyomino is responsible for this job. Three more types of polyominoes, the \textit{wires}, \textit{fillers}, and \textit{teeth}, are needed to form the complete framework of Ollinger.

To prove Theorem \ref{thm_wang_bar_44}, Jeandel and Rolin followed Ollinger's framework almost exactly. All five kinds of polyominoes have their counterparts in the proof of Jeandel and Rolin, in forms of \textit{gourps of Wang bars}. The counterparts for the meat, jaw, wires, and fillers are called the content, box, handles, and fillers, respectively by Jeandel and Rolin. In particular, the box can embrace a content from $3$ out of $4$ main directions (east, south, west, and north), which resembles the mechanism that the jaw can embrace a meat from $3$ directions. The counterparts for the fifth kind of polyominoes, the teeth, are unnamed by Jeandel and Rolin. More strictly speaking, Jeandel and Rolin employed a sixth kind of group of Wang bars which has no counterpart in Ollinger's framework. They can be viewed as teeth outside the jaws to fill the gaps in terms of Ollinger's framework and terminology. 

In summary, both Ollinger's original framework and the variant of Jeandel and Rolin have a clear distinction between inside and outside areas concerning the jaws or boxes. In our proof of Theorem \ref{thm_main}, we still follow Ollinger's framework loosely. However, we do not have an exact counterpart to the jaw polyomino in our reduction, and there is no distinction between inside and outside. In our reduction, we will construct $6$ kinds of groups of Wang bars: the encoder, the selector, the aligner, the linkers, fillers of group $1$, and fillers of group $2$.

\subsection{Construction of Wang Bars}

\begin{proof}[Proof of Theorem \ref{thm_main}] 
For each set of Wang tiles, we will construct a set of $29$ Wang bars. By assigning unique colors to the top and bottom sides of some of the Wang bars, these $29$ Wang bars must be used in groups, forming $6$ kinds of groups of Wang bars, which will be introduced one by one in the following paragraphs. We will take the set of $3$ Wang tiles in Figure \ref{fig_wang_set} as a concrete example to describe our reduction, and our method can be applied to the general cases naturally.

The first group is the \textit{encoder} which is illustrated in Figure \ref{fig_encoder}. The encoder is a single Wang bar that encodes all the Wang tiles of a set. To encode the set of Wang tiles in Figure \ref{fig_wang_set}, our encoder has a length of $27$ consisting of $3$ \textit{sections} of length $9$. Each section simulates a Wang tile. An enlarged view of the left-most section is also depicted in Figure \ref{fig_encoder}, which simulates a Wang tile tilted by $\pi/4$. The square in the middle of each section (illustrated in dark orange in Figure \ref{fig_encoder}) is called a \textit{locator}. The top and bottom sides of each locator are assigned to colors labeled $a_1$ and $a_2$, respectively. The locator separates each section into $2$ parts of length $4$, called the \textit{left arm} and \textit{right arm}. For the left-most section, the $4$ unit segments of the top side of the left arm are assigned to colors either $a_3$ and $a_5$, and exactly one unit segment is assigned to $a_3$. Because there are exactly $4$ different colors in the set of Wang tiles in Figure \ref{fig_wang_set}, the position of the color $a_3$ encodes the $4$ colors in a unary way. More precisely, reading the color sequences from left to right, $a_3a_5a_5a_5$, $a_5a_3a_5a_5$, $a_5a_5a_3a_5$ and $a_5a_5a_5a_3$ encode red, green, blue and yellow, respectively. The top side of the right arm is colored by either $a_3$ or $a_5$ in exactly the same way as the top side of the left arm. The bottom sides of the left arm and right arm are colored in a way almost the same as the top sides, except that the colors $a_3$ and $a_5$ are replaced by $a_4$ and $a_6$, respectively. Because the positions of $a_3$ and $a_4$ encode the colors of Wang tiles, we called them \textit{encoding colors}. Finally, the left and right sides of the encoder are colored $z$.


\begin{figure}[H]
\begin{center}
\begin{tikzpicture}[scale=0.5]

\foreach \x in {30}
\foreach \y in {5.5}
{
\draw [ fill=green!80] (\x+0,0+\y)--(\x+2,2+\y)--(\x+2,0+\y)--(\x+0,0+\y);
\draw [ fill=red!80] (\x+0,0+\y)--(\x+2,-2+\y)--(\x+2,0+\y)--(\x+0,0+\y);
\draw [ fill=red!80] (\x+4,0+\y)--(\x+2,2+\y)--(\x+2,0+\y)--(\x+4,0+\y);
\draw [ fill=yellow!80] (\x+4,0+\y)--(\x+2,-2+\y)--(\x+2,0+\y)--(\x+4,0+\y);

}

\draw [fill=orange!10] (0,0)--(27,0)--(27,1)--(0,1)--(0,0);

\foreach \x in {1,...,26}
{
\draw (\x,1)--(\x,0);
}

\foreach \x in {4,13,22}
{
\draw [fill=orange!80] (\x,1)--(\x,0)--(\x+1,0)--(\x+1,1)--(\x,1);
}

\draw [fill=orange!10]  (0,4)--(27,4)--(27,7)--(0,7)--(0,4);

\foreach \x in {1,...,8}
{
\draw (3*\x,4)--(3*\x,7);
}

\draw [<->, line width=3, ] (27.5,5.5)--(29.5,5.5);

\draw [very thick] (0,1.1)--(0,3.9);
\draw [very thick] (9,1.1)--(27,3.9);

\draw (0,4)--(1.5,5.5)--(27,5.5);
\draw (0,7)--(1.5,5.5); 

\foreach \x in {0,2,3,6,7,8}
{
\node at (3*\x+1.5,6.5) {$a_5$};
}

\foreach \x in {1,5}
{
\node at (3*\x+1.5,6.5) {$a_3$};
}
\foreach \x in {0,8}
{
\node at (3*\x+1.5,4.5) {$a_4$};
}

\foreach \x in {1,2,3,5,6,7}
{
\node at (3*\x+1.5,4.5) {$a_6$};
}

\node at (13.5,4.5) {$a_2$};
\node at (13.5,6.5) {$a_1$};

\node at (0.5,5.5) {$z$};

\draw [orange, very thick] (12,4)--(15,4)--(15,7)--(12,7)--(12,4);
\draw [red, very thick] (15,7)--(27,7);
\draw [red, very thick] (0,4)--(12,4);
\draw [yellow, very thick] (15,4)--(27,4);
\draw [green, very thick] (0,7)--(12,7);

\draw (27,0)--(26.5,0.5)--(27,1); \node at (27.3,0.5) {$z$};
\draw (0,0)--(0.5,0.5)--(0,1); \node at (-0.3,0.5) {$z$};
\draw (0.5,0.5)--(26.5,0.5);

\end{tikzpicture}
\end{center}
\caption{The encoder.}\label{fig_encoder}
\end{figure}
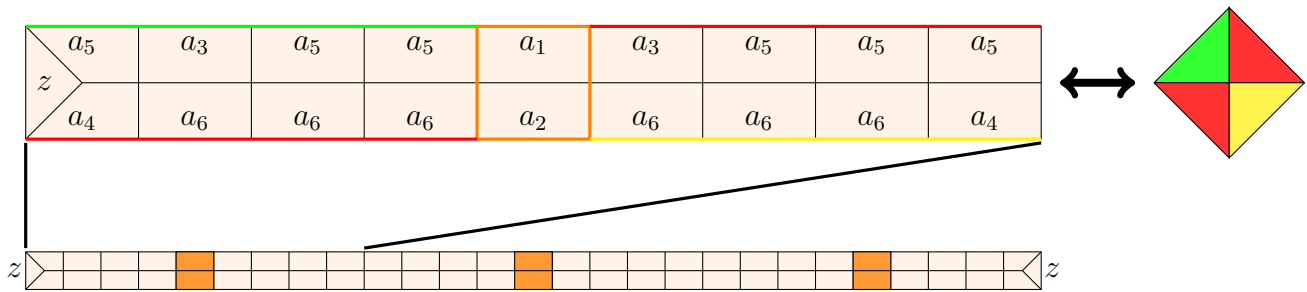

In general, for a set of $n$ Wang tiles with $m$ different colors, we need to construct an encoder Wang bar of length $n(2m+1)$, consisting of $n$ sections of length $2m+1$.

The second group, called \textit{selector}, consists of $7$ Wang bars as illustrated in Figure \ref{fig_selector}. One of the $7$ Wang bars has length $9$, the same as the length of a section of the encoder, and all the other $6$ Wang bars have length $1$. Each pair of the $6$ internal unit segments (illustrated in dark blue in Figure \ref{fig_encoder}) are assigned a unique color, hence the $7$ Wang bars must appear in groups. The top Wang bar and bottom Wang bar of the selector are called \textit{top pointer} and \textit{bottom pointer}, respectively. The top side of the top pointer is colored $a_2$, and the bottom side of the bottom pointer is colored $a_1$. Like the sections of the encoder, the two parts separated by the central square of the longest Wang bar of the selector are called two \textit{arms}. The top sides of the two arms are colored $a_7$, and the bottom side of them are colored $a_8$.

The leftmost and rightmost sides of the selector are colored $z$. The left sides of the second and sixth Wang bars (counting from the top) are colored $x$, and their right sides are colored $y$. All other unlabelled vertical sides (i.e. vertical sides of the first, third, fifth, and seventh Wang bars of the selector) are colored $0$, and this convention also applies to Figure \ref{fig_linker}, Figure \ref{fig_filler_1} and Figure \ref{fig_filler_2}. 

In the general case, we extend the length of arms of the selector from $4$ in this example to $m$, where $m$ is the number of different colors of a set of Wang tiles. In other words, the long Wang bar of the selector has length $2m+1$.


\begin{figure}[H]
\begin{center}
\begin{tikzpicture}[scale=1.5]

\draw [fill=blue!10] (0,0)--(4,0)--(4,-3)--(5,-3)--(5,0)--(9,0)--(9,1)--(5,1)--(5,4)--(4,4)--(4,1)--(0,1)--(0,0);

\draw (0,0)--(9,0)--(9,1)--(0,1)--(0,0);

\foreach \x in {1,...,8}
{
\draw (\x,1)--(\x,0);
}

\foreach \y in {1,2,3}
{
\draw (4,\y)--(4,\y+1)--(5,\y+1)--(5,\y);
}

\foreach \y in {0,-1,-2}
{
\draw (4,\y)--(4,\y-1)--(5,\y-1)--(5,\y);
}

\foreach \x in {0,1,2,3,5,6,7,8}
{
\node at (\x+0.5,0.2) {$a_8$};
\node at (\x+0.5,0.8) {$a_7$};
}
\draw (0,0)--(0.5,0.5)--(8.5,0.5)--(9,0); \draw (0,1)--(0.5,0.5); \draw (8.5,0.5)--(9,1);
\node at (0.2,0.5) {$z$}; \node at (8.8,0.5) {$z$};

\draw (4,-3)--(4.5,-2.5)--(5,-3); \draw (4,4)--(4.5,3.5)--(5,4);
\node at (4.5,-2.8) {$a_1$}; \node at (4.5,3.8) {$a_2$};

\foreach \y in {-2,2}
{
\draw (4,\y)--(5,\y+1);
\draw (4,\y+1)--(5,\y);
\node at (4.2,\y+0.5) {$x$};
\node at (4.8,\y+0.5) {$y$};
}

\foreach \y in {-2,...,3}
{
\draw [color=blue, very thick] (4,\y)--(5,\y);
}

\end{tikzpicture}
\end{center}
\caption{The selector.}\label{fig_selector}
\end{figure}
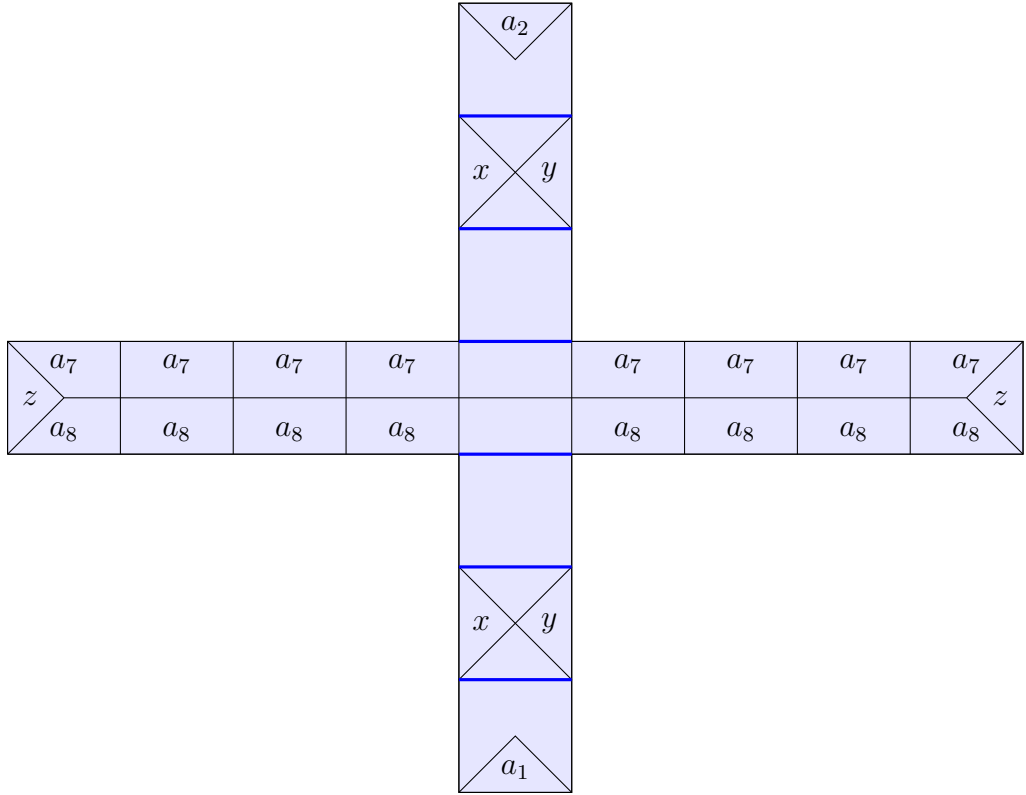

As we will see later, the bottom pointer of a selector can only be adjacent to a locator of an encoder from above, and the top pointer can only be adjacent to a locator from below. In this way, the selector chooses one section (a simulated Wang tile) from every encoder.

The third group is the \textit{aligner} (see Figure \ref{fig_aligner}), which consists of a single Wang bar of length $9$. In the general case, the aligner has length $2m+1$. All unit segments of the top side are colored $a_5$, and all unit segments of the bottom side are colored $a_6$. Both the left and right sides are colored $z$.


\begin{figure}[H]
\begin{center}
\begin{tikzpicture}[scale=1.5]

\draw [fill=yellow!10] (0,0)--(9,0)--(9,1)--(0,1)--(0,0);

\foreach \x in {1,...,8}
{
\draw (\x,1)--(\x,0);
}

\draw (0,0)--(0.5,0.5)--(8.5,0.5)--(9,0); \draw (0,1)--(0.5,0.5); \draw (8.5,0.5)--(9,1);
\node at (0.2,0.5) {$z$}; \node at (8.8,0.5) {$z$};

\foreach \x in {0,...,8}
{
\node at (\x+0.5,0.2) {$a_6$};
\node at (\x+0.5,0.8) {$a_5$};
}

\end{tikzpicture}
\end{center}
\caption{The aligner.}\label{fig_aligner}
\end{figure}
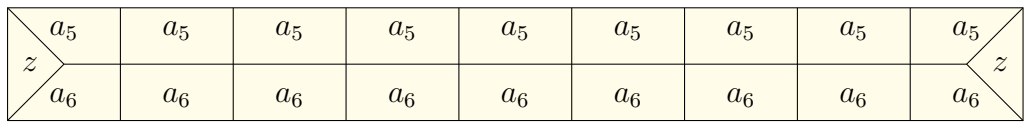

There are several features in common in the three groups of Wang bars, the encoder, selector, and aligner, we have introduced so far. First, they all have a length of $2m+1$, or a multiple of $2m+1$ in the general case. Second, the color $z$ only appears on the left or right sides in these three groups. Third, the top and bottom sides are all colored by $a_i$ $(1\leq i \leq 8)$. Together they will form a \textit{backbone} structure of the overall tiling (as illustrated in Figure \ref{fig_pattern} later).

The fourth group (the \textit{linkers}) is illustrated in Figure \ref{fig_linker}. There are $4$ different Wang bars in this group. Two of them are of length $1$ and the other two are of length $23$. In the general case, the short ones remain the same, while the long ones are of length $(n-1)(2m+1)+m+1$. Because the colors $b_1$ and $b_2$ only appear in these four Wang bars, they must form in groups in two different ways as illustrated in Figure \ref{fig_linker}.


\begin{figure}[H]
\begin{center}
\begin{tikzpicture}[scale=1.5]

\draw [fill=gray!10] (0,5)--(3.3,5)--(3.6,5.6)--(3.3,6)--(1,6)--(1,7)--(0,7)--(0,5); 
\draw [fill=gray!10] (3.5,5)--(6,5)--(6,4)--(7,4)--(7,6)--(3.5,6)--(3.8,5.6)--(3.5,5);
\draw (7,4)--(6,5)--(6,4)--(7,5)--(7,4)--(6,4); \draw (6,5)--(7,5);
\draw (0,7)--(1,7)--(0,6)--(0,7)--(1,6)--(1,7); \draw (1,6)--(0,6);

\foreach \x in {1,...,6}
{
\draw (\x,5)--(\x,6);
}


\begin{scope}
    \clip (0,5)--(3.3,5)--(3.6,5.6)--(3.3,6)--(0,6)--(0,5); 
    \draw (0,5)--(0.5,5.5)--(4,5.5);
\end{scope}
\draw (0,6)--(0.5,5.5);

\begin{scope}
    \clip (3.5,5)--(7,5)--(7,6)--(3.5,6)--(3.8,5.6)--(3.5,5);
    \draw (3,5.5)--(6.5,5.5)--(7,5);
\end{scope}
\draw (6.5,5.5)--(7,6);

\foreach \x in {1,2,2.7,3.3,4,5,6}
{
\node at (\x+0.5,5.8) {$b_3$};
}
\foreach \x in {0}
{
\node at (\x+0.5,5.8) {$b_1$};
\node at (\x+0.5,6.2) {$b_1$};
\node at (\x+0.5,6.8) {$a_4$};
\node at (\x+0.2,5.5) {$y$};
}

\foreach \x in {0,1,2,3.2-0.5,3.8-0.5,4,5}
{
\node at (\x+0.5,5.2) {$b_4$};
}
\foreach \x in {6}
{
\node at (\x+0.5,5.2) {$b_2$};
\node at (\x+0.5,4.8) {$b_2$};
\node at (\x+0.5,4.2) {$a_3$};
\node at (\x+0.8,5.5) {$x$};
}


\draw [fill=gray!10] (0,0)--(1,0)--(1,1)--(3.3,1)--(3.6,1.6)--(3.3,2)--(0,2)--(0,0);

\draw [fill=gray!10] (3.5,1)--(7,1)--(7,3)--(6,3)--(6,2)--(3.5,2)--(3.8,1.6)--(3.5,1);

\foreach \x in {1,...,6}
{
\draw (\x,1)--(\x,2);
}
\draw (0,1)--(1,1); \draw (6,2)--(7,2);


\begin{scope}
    \clip (0,1)--(3.3,1)--(3.6,1.6)--(3.3,2)--(0,2)--(0,1);   
    \draw (0,1)--(0.5,1.5)--(4,1.5);
\end{scope}

\begin{scope}
    \clip (3.5,1)--(7,1)--(7,2)--(3.5,2)--(3.8,1.6)--(3.5,1);  
    \draw (3,1.5)--(6.5,1.5)--(7,1);
\end{scope}

\draw (6.5,1.5)--(7,2); \draw (0.5,1.5)--(0,2);

\foreach \x in {0,1,2,2.7,3.3,4,5}
{
\node at (\x+0.5,1.8) {$b_3$};
}
\foreach \x in {6}
{
\node at (\x+0.5,1.8) {$b_1$};
\node at (\x+0.5,2.2) {$b_1$};
\node at (\x+0.5,2.8) {$a_4$};
}

\foreach \x in {1,2,3.2-0.5,3.8-0.5,4,5,6}
{
\node at (\x+0.5,1.2) {$b_4$};
}
\foreach \x in {0}
{
\node at (\x+0.5,1.2) {$b_2$};
\node at (\x+0.5,0.8) {$b_2$};
\node at (\x+0.5,0.2) {$a_3$};
}

\draw (0,0)--(1,1); \draw (1,0)--(0,1);
\draw (6,2)--(7,3); \draw (6,3)--(7,2);

\node at (0.2,1.5) {$y$};
\node at (6.8,1.5) {$x$};

\end{tikzpicture}
\end{center}
\caption{The linkers.}\label{fig_linker}
\end{figure}
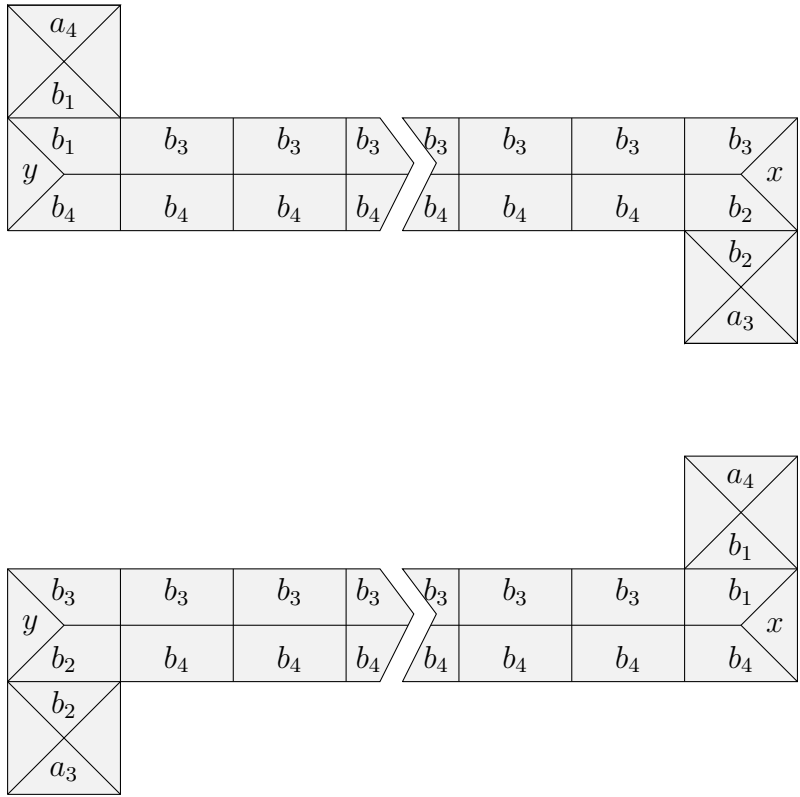

The top and bottom sides of the first Wang bar of length $1$ are colored $a_4$ and $b_1$, respectively. The top and bottom sides of the second Wang bar of length $1$ are colored $b_2$ and $a_3$, respectively. The vertical sides of the two short Wang bars are colored $0$, their labels are omitted in Figure \ref{fig_linker} by our convention.  The top and bottom sides of the first Wang bar of length $23$ are colored $b_1(b_3)^{22}$ and $(b_4)^{22}b_2$, respectively. The top and bottom sides of the second Wang bar of length $23$ are colored $(b_3)^{22}b_1$ and $b_2(b_4)^{22}$, respectively. For the two long Wang bars, the left sides are colored $y$, and the right sides are colored $x$.

So a linker connects an encoding color $a_3$ to another encoding color $a_4$. As we have mentioned earlier, the two colors $a_3$ and $a_4$ are critical in encoding the colors of Wang tiles in each section of an encoder. The linkers can connect two sections from two encoders only if the corresponding sides of the two sections encode the same color, as we will explain in detail later.

All the remaining Wang bars are primarily used to fill the gaps left by the previous four groups of Wang bars in forming a tiling of the plane. So we call them the \textit{fillers}. According to their slightly different roles in filling the gaps, we further divide them into two groups, which make up the fifth and sixth groups of our complete set of Wang bars. Note that the fillers remain unchanged in the general case. In other words, their construction does not depend on $n$ or $m$.


\begin{figure}[H]
\begin{center}
\begin{tikzpicture}[scale=1]

\foreach \x in {0,2,4,6, 9,11,13,15}
{
\draw[fill=green!10]  (\x,0)--(\x+1,0)--(\x+1,1)--(\x,1)--(\x,0);
\draw (\x,0)--(\x+1,1); \draw (\x,1)--(\x+1,0);
}
\foreach \x in {0,...,3}
{
\node at (2*\x+0.5,0.2) {$b_3$};
}

\foreach \x in {4,...,7}
{
\node at (2*\x+1.5,0.8) {$b_4$};
}

\node at (0.5,0.8) {$a_2$};
\node at (2.5,0.8) {$a_4$};
\node at (4.5,0.8) {$a_6$};
\node at (6.5,0.8) {$a_8$};

\node at (9.5,0.2) {$a_1$};
\node at (11.5,0.2) {$a_3$};
\node at (13.5,0.2) {$a_5$};
\node at (15.5,0.2) {$a_7$};

\end{tikzpicture}
\end{center}
\caption{The fillers (Group 1).}\label{fig_filler_1}
\end{figure}
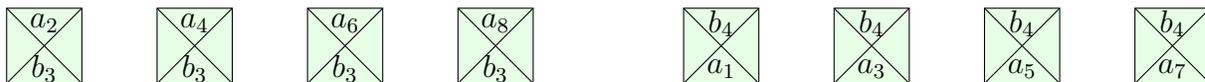

The fifth group (group $1$ of the fillers) is illustrated in Figure \ref{fig_filler_1}. There are $8$ Wang bars of length $1$ in this group. The bottom sides of the first $4$ Wang bars are colored $b_3$, and their top sides are $a_2$, $a_4$, $a_6$, and $a_8$, respectively. The top sides of the rest $4$ Wang bars are $b_4$, and their bottom sides are $a_1$, $a_3$, $a_5$ and $a_7$, respectively. The vertical sides of all the $8$ Wang bars are $0$. This group is used to fill the gaps between the linkers and the backbone (recall that the backbone is formed by the encoders, selectors, and aligners).


\begin{figure}[H]
\begin{center}
\begin{tikzpicture}[scale=1]

\foreach \x in {0,2,5,7}
\foreach \y in {0}
{
\draw [fill=red!10] (\x,\y)--(\x+1,\y)--(\x+1,\y+3)--(\x,\y+3)--(\x,\y);
}

\foreach \x in {0,2,5,7}
\foreach \y in {0,1,2}
{
\draw (\x,\y)--(\x+1,\y)--(\x+1,\y+1)--(\x,\y+1)--(\x,\y);
}

\foreach \x in {0,2,5,7}
\foreach \y in {0,2}
{
\draw (\x,\y+1)--(\x+1,\y); \draw (\x,\y)--(\x+1,\y+1);
}
\foreach \x in {0,5}
{
\draw (\x,1)--(\x+1,2); 
\draw (\x,2)--(\x+1,1);
}

\foreach \x in {2,7}
\foreach \y in {1}
{
\draw (\x,\y)--(\x+1,\y)--(\x+1,\y+1)--(\x,\y+1)--(\x,\y);
\draw (\x,1)--(\x+1,2); 
\draw (\x,2)--(\x+1,1);
}


\node at (0.5,2.2) {$b_5$};
\node at (0.5,1.8) {$b_5$};
\node at (0.5,1.2) {$b_6$};
\node at (0.5,0.8) {$b_6$};
\node at (0.5,0.2) {$a_5$};
\node at (0.5,2.8) {$a_8$};

\node at (0.2,1.5) {$x$};
\node at (0.8,1.5) {$x$};

\node at (5.5,2.2) {$b_7$};
\node at (5.5,1.8) {$b_7$};
\node at (5.5,1.2) {$b_8$};
\node at (5.5,0.8) {$b_8$};
\node at (5.5,0.2) {$a_7$};
\node at (5.5,2.8) {$a_6$};

\node at (5.2,1.5) {$x$};
\node at (5.8,1.5) {$x$};

\node at (2.5,0.8) {$b_6$};
\node at (2.5,0.2) {$a_5$};
\node at (2.5,2.8) {$a_8$};
\node at (2.5,2.2) {$b_5$};
\node at (2.5,1.8) {$b_5$};
\node at (2.5,1.2) {$b_6$};
\node at (2.2,1.5) {$y$};
\node at (2.8,1.5) {$y$};

\node at (7.5,2.2) {$b_7$};
\node at (7.5,0.8) {$b_8$};
\node at (7.5,0.2) {$a_7$};
\node at (7.5,2.8) {$a_6$};
\node at (7.5,1.8) {$b_7$};
\node at (7.5,1.2) {$b_8$};
\node at (7.2,1.5) {$y$};
\node at (7.8,1.5) {$y$};

\end{tikzpicture}
\end{center}
\caption{The fillers (Group 2).}\label{fig_filler_2}
\end{figure}
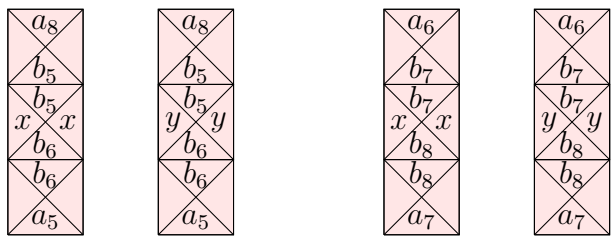

The last group (group $2$ of the fillers) is illustrated in Figure \ref{fig_filler_2}. There are also $8$ different Wang bars of length $1$ in this group. But unlike the previous group, the Wang bars in this group must appear in batches of three forming vertical rectangles of height $3$. The internal colors $b_5$, $b_6$, $b_7$, and $b_8$ enforce the $8$ Wang bars to form $4$ kinds of vertical rectangles. Two of the vertical rectangles are colored $a_8$ on the top side and $a_5$ at the bottom side. They differ only on the left and right sides of the middle Wang bar, where they are colored either both $x$ or both $y$. The other two vertical rectangles are colored $a_6$ on the top side and $a_7$ at the bottom side. The left and right sides of their middle Wang bar are also colored either both $x$ or both $y$. All other unlabelled sides in Figure \ref{fig_filler_2} are colored $0$. This group fills the gaps between the backbones.

\subsection{Tiling Pattern}

We have finished introducing all the $6$ groups of Wang bars, which add up to a total of $29$ Wang bars. To complete the reduction, we need to show that to tile the plane with this set of $29$ Wang bars, we must follow the pattern illustrated in Figure \ref{fig_pattern}. In Figure \ref{fig_pattern}, the tiles in orange, blue, yellow, gray, green, and red represent the encoders, selectors, aligners, linkers, fillers of group $1$ and fillers of group $2$, respectively. In this pattern, exactly one section (an emulated Wang tile) of each encoder is selected by the pointers of two selectors from above and below. The selected section form a lattice in the plane. Every selected sections must match its four neighboring selected sections by the linkers. Overall, this pattern simulates a tiling of a set of Wang tiles. The pattern in Figure \ref{fig_pattern} is equivalent to the partial tiling of Wang tiles illustrated in Figure \ref{fig_wang_region}, where all the Wang tiles are tilted by $\pi/4$. 


\begin{figure}[H]
\begin{center}
\begin{tikzpicture}[scale=0.2]

\foreach \x in {0,54}
\foreach \y in {0,8}
{
\draw [fill=blue!30] (\x+0,0+\y)--(\x+0,-3+\y)--(\x+1,-3+\y)--(\x+1,0+\y)--(\x+5,0+\y)--(\x+5,1+\y)--(\x+1,1+\y)--(\x+1,4+\y)--(\x+0,4+\y)--(\x+0,1+\y)--(\x+-4,1+\y)--(\x+-4,0+\y)--(\x+0,0+\y);
}

\foreach \x in {27,81}
\foreach \y in {-4,4,12}
{
\draw [fill=blue!30] (\x+0,0+\y)--(\x+0,-3+\y)--(\x+1,-3+\y)--(\x+1,0+\y)--(\x+5,0+\y)--(\x+5,1+\y)--(\x+1,1+\y)--(\x+1,4+\y)--(\x+0,4+\y)--(\x+0,1+\y)--(\x+-4,1+\y)--(\x+-4,0+\y)--(\x+0,0+\y);
}


\foreach \x in {-18}
\foreach \y in {0}
{
\draw [fill=orange!30] (\x+14,-3+\y)--(\x+23,-3+\y)--(\x+23,-4+\y)--(\x+14,-4+\y);
}

\foreach \x in {-9}
\foreach \y in {8}
{
\draw [fill=orange!30] (\x+5,-3+\y)--(\x+23,-3+\y)--(\x+23,-4+\y)--(\x+5,-4+\y);
}

\foreach \x in {0}
\foreach \y in {16}
{
\draw [fill=orange!30] (\x+-4,-3+\y)--(\x+23,-3+\y)--(\x+23,-4+\y)--(\x+-4,-4+\y)--(\x+-4,-3+\y);
}

\foreach \x in {9}
\foreach \y in {4}
{
\draw [fill=orange!30] (\x+-4,-3+\y)--(\x+23,-3+\y)--(\x+23,-4+\y)--(\x+-4,-4+\y)--(\x+-4,-3+\y);
}
\foreach \x in {18}
\foreach \y in {12}
{
\draw [fill=orange!30] (\x+-4,-3+\y)--(\x+23,-3+\y)--(\x+23,-4+\y)--(\x+-4,-4+\y)--(\x+-4,-3+\y);
}

\foreach \x in {36}
\foreach \y in {8}
{
\draw [fill=orange!30] (\x+-4,-3+\y)--(\x+23,-3+\y)--(\x+23,-4+\y)--(\x+-4,-4+\y)--(\x+-4,-3+\y);
}
\foreach \x in {54}
\foreach \y in {0}
{
\draw [fill=orange!30] (\x+-4,-3+\y)--(\x+23,-3+\y)--(\x+23,-4+\y)--(\x+-4,-4+\y)--(\x+-4,-3+\y);
}
\foreach \x in {45}
\foreach \y in {16}
{
\draw [fill=orange!30] (\x+-4,-3+\y)--(\x+23,-3+\y)--(\x+23,-4+\y)--(\x+-4,-4+\y)--(\x+-4,-3+\y);
}

\foreach \x in {81}
\foreach \y in {4}
{
\draw [fill=orange!30] (\x+5,-4+\y)--(\x+-4,-4+\y)--(\x+-4,-3+\y)--(\x+5,-3+\y);
}

\foreach \x in {63}
\foreach \y in {12}
{
\draw [fill=orange!30] (\x+-4,-3+\y)--(\x+23,-3+\y)--(\x+23,-4+\y)--(\x+-4,-4+\y)--(\x+-4,-3+\y);
}


\foreach \x in {3}
\foreach \y in {0}
{
\draw [fill=gray!30] (\x+1,-3+\y)--(\x+2,-3+\y)--(\x+2,-2+\y)--(\x+24,-2+\y)--(\x+24,0+\y)--(\x+23,0+\y)--(\x+23,-1+\y)--(\x+1,-1+\y)--(\x+1,-3+\y);
}
\foreach \x in {30}
\foreach \y in {4}
{
\draw [fill=gray!30] (\x+1,-3+\y)--(\x+2,-3+\y)--(\x+2,-2+\y)--(\x+24,-2+\y)--(\x+24,0+\y)--(\x+23,0+\y)--(\x+23,-1+\y)--(\x+1,-1+\y)--(\x+1,-3+\y);
}
\foreach \x in {57}
\foreach \y in {8}
{
\draw [fill=gray!30] (\x+1,-3+\y)--(\x+2,-3+\y)--(\x+2,-2+\y)--(\x+24,-2+\y)--(\x+24,0+\y)--(\x+23,0+\y)--(\x+23,-1+\y)--(\x+1,-1+\y)--(\x+1,-3+\y);
}

\foreach \x in {2}
\foreach \y in {8}
{
\draw [fill=gray!30] (\x+1,-3+\y)--(\x+2,-3+\y)--(\x+2,-2+\y)--(\x+24,-2+\y)--(\x+24,0+\y)--(\x+23,0+\y)--(\x+23,-1+\y)--(\x+1,-1+\y)--(\x+1,-3+\y);
}
\foreach \x in {29}
\foreach \y in {12}
{
\draw [fill=gray!30] (\x+1,-3+\y)--(\x+2,-3+\y)--(\x+2,-2+\y)--(\x+24,-2+\y)--(\x+24,0+\y)--(\x+23,0+\y)--(\x+23,-1+\y)--(\x+1,-1+\y)--(\x+1,-3+\y);
}

\foreach \x in {54}
\foreach \y in {0}
{
\draw [fill=gray!30] (\x+1,-3+\y)--(\x+2,-3+\y)--(\x+2,-2+\y)--(\x+24,-2+\y)--(\x+24,0+\y)--(\x+23,0+\y)--(\x+23,-1+\y)--(\x+1,-1+\y)--(\x+1,-3+\y);
}


\foreach \x in {28}
\foreach \y in {0}
{
\draw [fill=gray!30] (\x+1,-2+\y)--(\x+23,-2+\y)--(\x+23,-3+\y)--(\x+24,-3+\y)--(\x+24,-1+\y)--(\x+2,-1+\y)--(\x+2,0+\y)--(\x+1,0+\y)--(\x+1,-2+\y);
}
\foreach \x in {55}
\foreach \y in {4}
{
\draw [fill=gray!30] (\x+1,-2+\y)--(\x+23,-2+\y)--(\x+23,-3+\y)--(\x+24,-3+\y)--(\x+24,-1+\y)--(\x+2,-1+\y)--(\x+2,0+\y)--(\x+1,0+\y)--(\x+1,-2+\y);
}

\foreach \x in {0}
\foreach \y in {4}
{
\draw [fill=gray!30] (\x+1,-2+\y)--(\x+23,-2+\y)--(\x+23,-3+\y)--(\x+24,-3+\y)--(\x+24,-1+\y)--(\x+2,-1+\y)--(\x+2,0+\y)--(\x+1,0+\y)--(\x+1,-2+\y);
}
\foreach \x in {27}
\foreach \y in {8}
{
\draw [fill=gray!30] (\x+1,-2+\y)--(\x+23,-2+\y)--(\x+23,-3+\y)--(\x+24,-3+\y)--(\x+24,-1+\y)--(\x+2,-1+\y)--(\x+2,0+\y)--(\x+1,0+\y)--(\x+1,-2+\y);
}
\foreach \x in {54}
\foreach \y in {12}
{
\draw [fill=gray!30] (\x+1,-2+\y)--(\x+23,-2+\y)--(\x+23,-3+\y)--(\x+24,-3+\y)--(\x+24,-1+\y)--(\x+2,-1+\y)--(\x+2,0+\y)--(\x+1,0+\y)--(\x+1,-2+\y);
}

\foreach \x in {3}
\foreach \y in {12}
{
\draw [fill=gray!30] (\x+1,-2+\y)--(\x+23,-2+\y)--(\x+23,-3+\y)--(\x+24,-3+\y)--(\x+24,-1+\y)--(\x+2,-1+\y)--(\x+2,0+\y)--(\x+1,0+\y)--(\x+1,-2+\y);
}

\foreach \x in {0,9,27,36}
\foreach \y in {0}
{
\draw [fill=yellow!30] (\x+5,-3+\y)--(\x+14,-3+\y)--(\x+14,-4+\y)--(\x+5,-4+\y)--(\x+5,-3+\y);
}
\foreach \x in {27,36,54,63}
\foreach \y in {4}
{
\draw [fill=yellow!30] (\x+5,-3+\y)--(\x+14,-3+\y)--(\x+14,-4+\y)--(\x+5,-4+\y)--(\x+5,-3+\y);
}
\foreach \x in {9,54,63}
\foreach \y in {8}
{
\draw [fill=yellow!30] (\x+5,-3+\y)--(\x+14,-3+\y)--(\x+14,-4+\y)--(\x+5,-4+\y)--(\x+5,-3+\y);
}
\foreach \x in {0,36}
\foreach \y in {12}
{
\draw [fill=yellow!30] (\x+5,-3+\y)--(\x+14,-3+\y)--(\x+14,-4+\y)--(\x+5,-4+\y)--(\x+5,-3+\y);
}
\foreach \x in {27,63}
\foreach \y in {16}
{
\draw [fill=yellow!30] (\x+5,-3+\y)--(\x+14,-3+\y)--(\x+14,-4+\y)--(\x+5,-4+\y)--(\x+5,-3+\y);
}

\foreach \x in {0,1,2,27,51,52,77,78,79}
\foreach \y in {0}
{
\draw [fill=red!30] (\x+1,\y)--(\x+2,\y)--(\x+2,-3+\y)--(\x+1,-3+\y)--(\x+1,\y);
}

\foreach \x in {23,24,25,27,28,29,54,78,79}
\foreach \y in {4}
{
\draw [fill=red!30] (\x+1,\y)--(\x+2,\y)--(\x+2,-3+\y)--(\x+1,-3+\y)--(\x+1,\y);
}

\foreach \x in {0,1,25,50,51,52,54,55,56}
\foreach \y in {8}
{
\draw [fill=red!30] (\x+1,\y)--(\x+2,\y)--(\x+2,-3+\y)--(\x+1,-3+\y)--(\x+1,\y);
}

\foreach \x in {0,1,2,27,28,52,77,78,79}
\foreach \y in {12}
{
\draw [fill=red!30] (\x+1,\y)--(\x+2,\y)--(\x+2,-3+\y)--(\x+1,-3+\y)--(\x+1,\y);
}

\foreach \x in {0,...,21}
\foreach \y in {0,14}
{
\draw [fill=green!30] (\x+5,-3+\y)--(\x+6,-3+\y)--(\x+6,-2+\y)--(\x+5,-2+\y)--(\x+5,-3+\y);
}
\foreach \x in {24,...,45}
\foreach \y in {0,10}
{
\draw [fill=green!30] (\x+5,-3+\y)--(\x+6,-3+\y)--(\x+6,-2+\y)--(\x+5,-2+\y)--(\x+5,-3+\y);
}
\foreach \x in {51,...,72}
\foreach \y in {0,4,14}
{
\draw [fill=green!30] (\x+5,-3+\y)--(\x+6,-3+\y)--(\x+6,-2+\y)--(\x+5,-2+\y)--(\x+5,-3+\y);
}

\foreach \x in {-1,...,20}
\foreach \y in {2,8,12}
{
\draw [fill=green!30] (\x+5,-3+\y)--(\x+6,-3+\y)--(\x+6,-2+\y)--(\x+5,-2+\y)--(\x+5,-3+\y);
}
\foreach \x in {25,...,46}
\foreach \y in {2,14}
{
\draw [fill=green!30] (\x+5,-3+\y)--(\x+6,-3+\y)--(\x+6,-2+\y)--(\x+5,-2+\y)--(\x+5,-3+\y);
}
\foreach \x in {50,...,71}
\foreach \y in {2,12}
{
\draw [fill=green!30] (\x+5,-3+\y)--(\x+6,-3+\y)--(\x+6,-2+\y)--(\x+5,-2+\y)--(\x+5,-3+\y);
}

\foreach \x in {-4,...,17}
\foreach \y in {4}
{
\draw [fill=green!30] (\x+5,-3+\y)--(\x+6,-3+\y)--(\x+6,-2+\y)--(\x+5,-2+\y)--(\x+5,-3+\y);
}
\foreach \x in {27,...,48}
\foreach \y in {4}
{
\draw [fill=green!30] (\x+5,-3+\y)--(\x+6,-3+\y)--(\x+6,-2+\y)--(\x+5,-2+\y)--(\x+5,-3+\y);
}

\foreach \x in {-3,...,18}
\foreach \y in {6}
{
\draw [fill=green!30] (\x+5,-3+\y)--(\x+6,-3+\y)--(\x+6,-2+\y)--(\x+5,-2+\y)--(\x+5,-3+\y);
}
\foreach \x in {26,...,47}
\foreach \y in {6,12}
{
\draw [fill=green!30] (\x+5,-3+\y)--(\x+6,-3+\y)--(\x+6,-2+\y)--(\x+5,-2+\y)--(\x+5,-3+\y);
}
\foreach \x in {52,...,73}
\foreach \y in {6}
{
\draw [fill=green!30] (\x+5,-3+\y)--(\x+6,-3+\y)--(\x+6,-2+\y)--(\x+5,-2+\y)--(\x+5,-3+\y);
}

\foreach \x in {23,...,44}
\foreach \y in {8}
{
\draw [fill=green!30] (\x+5,-3+\y)--(\x+6,-3+\y)--(\x+6,-2+\y)--(\x+5,-2+\y)--(\x+5,-3+\y);
}
\foreach \x in {54,...,75}
\foreach \y in {8}
{
\draw [fill=green!30] (\x+5,-3+\y)--(\x+6,-3+\y)--(\x+6,-2+\y)--(\x+5,-2+\y)--(\x+5,-3+\y);
}

\foreach \x in {-2,...,19}
\foreach \y in {10}
{
\draw [fill=green!30] (\x+5,-3+\y)--(\x+6,-3+\y)--(\x+6,-2+\y)--(\x+5,-2+\y)--(\x+5,-3+\y);
}
\foreach \x in {53,...,74}
\foreach \y in {10}
{
\draw [fill=green!30] (\x+5,-3+\y)--(\x+6,-3+\y)--(\x+6,-2+\y)--(\x+5,-2+\y)--(\x+5,-3+\y);
}

\end{tikzpicture}
\end{center}
\caption{The tiling pattern.}\label{fig_pattern}
\end{figure}
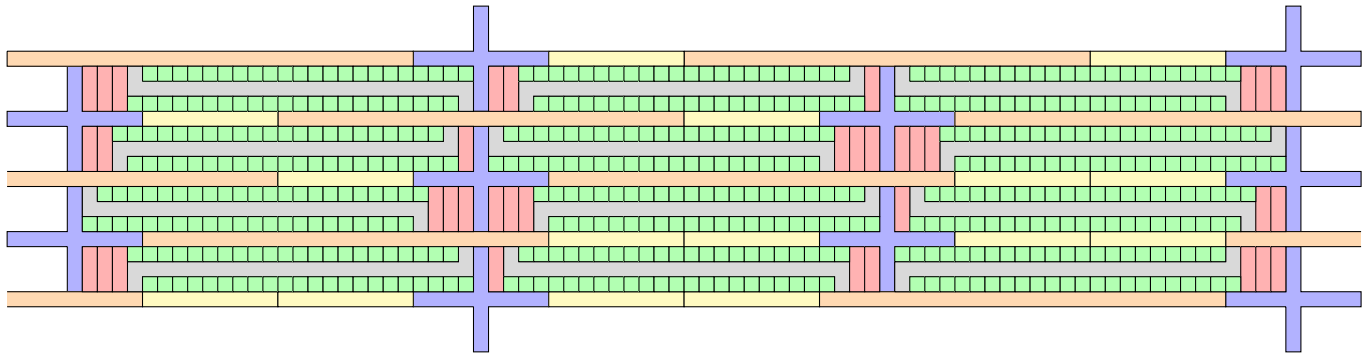


\begin{figure}[H]
\begin{center}
\begin{tikzpicture}[scale=0.3]

\foreach \x in {0}
\foreach \y in {0}
{
\draw [ fill=green!80] (\x+0,0+\y)--(\x+2,2+\y)--(\x+2,0+\y)--(\x+0,0+\y);
\draw [ fill=red!80] (\x+0,0+\y)--(\x+2,-2+\y)--(\x+2,0+\y)--(\x+0,0+\y);
\draw [ fill=red!80] (\x+4,0+\y)--(\x+2,2+\y)--(\x+2,0+\y)--(\x+4,0+\y);
\draw [ fill=yellow!80] (\x+4,0+\y)--(\x+2,-2+\y)--(\x+2,0+\y)--(\x+4,0+\y);
}
\foreach \x in {4}
\foreach \y in {-8}
{
\draw [ fill=green!80] (\x+0,0+\y)--(\x+2,2+\y)--(\x+2,0+\y)--(\x+0,0+\y);
\draw [ fill=red!80] (\x+0,0+\y)--(\x+2,-2+\y)--(\x+2,0+\y)--(\x+0,0+\y);
\draw [ fill=red!80] (\x+4,0+\y)--(\x+2,2+\y)--(\x+2,0+\y)--(\x+4,0+\y);
\draw [ fill=yellow!80] (\x+4,0+\y)--(\x+2,-2+\y)--(\x+2,0+\y)--(\x+4,0+\y);
}
\foreach \x in {6}
\foreach \y in {-6}
{
\draw [ fill=green!80] (\x+0,0+\y)--(\x+2,2+\y)--(\x+2,0+\y)--(\x+0,0+\y);
\draw [ fill=red!80] (\x+0,0+\y)--(\x+2,-2+\y)--(\x+2,0+\y)--(\x+0,0+\y);
\draw [ fill=red!80] (\x+4,0+\y)--(\x+2,2+\y)--(\x+2,0+\y)--(\x+4,0+\y);
\draw [ fill=yellow!80] (\x+4,0+\y)--(\x+2,-2+\y)--(\x+2,0+\y)--(\x+4,0+\y);
}

\foreach \x in {2}
\foreach \y in {-2}
{
\draw [ fill=yellow!80] (\x+0,0+\y)--(\x+2,2+\y)--(\x+2,0+\y)--(\x+0,0+\y);
\draw [ fill=blue!80] (\x+0,0+\y)--(\x+2,-2+\y)--(\x+2,0+\y)--(\x+0,0+\y);
\draw [ fill=blue!80] (\x+4,0+\y)--(\x+2,2+\y)--(\x+2,0+\y)--(\x+4,0+\y);
\draw [ fill=red!80] (\x+4,0+\y)--(\x+2,-2+\y)--(\x+2,0+\y)--(\x+4,0+\y);
}
\foreach \x in {4}
\foreach \y in {0}
{
\draw [ fill=yellow!80] (\x+0,0+\y)--(\x+2,2+\y)--(\x+2,0+\y)--(\x+0,0+\y);
\draw [ fill=blue!80] (\x+0,0+\y)--(\x+2,-2+\y)--(\x+2,0+\y)--(\x+0,0+\y);
\draw [ fill=blue!80] (\x+4,0+\y)--(\x+2,2+\y)--(\x+2,0+\y)--(\x+4,0+\y);
\draw [ fill=red!80] (\x+4,0+\y)--(\x+2,-2+\y)--(\x+2,0+\y)--(\x+4,0+\y);
}
\foreach \x in {0}
\foreach \y in {-4}
{
\draw [ fill=yellow!80] (\x+0,0+\y)--(\x+2,2+\y)--(\x+2,0+\y)--(\x+0,0+\y);
\draw [ fill=blue!80] (\x+0,0+\y)--(\x+2,-2+\y)--(\x+2,0+\y)--(\x+0,0+\y);
\draw [ fill=blue!80] (\x+4,0+\y)--(\x+2,2+\y)--(\x+2,0+\y)--(\x+4,0+\y);
\draw [ fill=red!80] (\x+4,0+\y)--(\x+2,-2+\y)--(\x+2,0+\y)--(\x+4,0+\y);
}

\foreach \x in {6}
\foreach \y in {-2}
{
\draw [ fill=red!80] (\x+0,0+\y)--(\x+2,2+\y)--(\x+2,0+\y)--(\x+0,0+\y);
\draw [ fill=yellow!80] (\x+0,0+\y)--(\x+2,-2+\y)--(\x+2,0+\y)--(\x+0,0+\y);
\draw [ fill=yellow!80] (\x+4,0+\y)--(\x+2,2+\y)--(\x+2,0+\y)--(\x+4,0+\y);
\draw [ fill=green!80] (\x+4,0+\y)--(\x+2,-2+\y)--(\x+2,0+\y)--(\x+4,0+\y);
}
\foreach \x in {4}
\foreach \y in {-4}
{
\draw [ fill=red!80] (\x+0,0+\y)--(\x+2,2+\y)--(\x+2,0+\y)--(\x+0,0+\y);
\draw [ fill=yellow!80] (\x+0,0+\y)--(\x+2,-2+\y)--(\x+2,0+\y)--(\x+0,0+\y);
\draw [ fill=yellow!80] (\x+4,0+\y)--(\x+2,2+\y)--(\x+2,0+\y)--(\x+4,0+\y);
\draw [ fill=green!80] (\x+4,0+\y)--(\x+2,-2+\y)--(\x+2,0+\y)--(\x+4,0+\y);
}
\foreach \x in {2}
\foreach \y in {-6}
{
\draw [ fill=red!80] (\x+0,0+\y)--(\x+2,2+\y)--(\x+2,0+\y)--(\x+0,0+\y);
\draw [ fill=yellow!80] (\x+0,0+\y)--(\x+2,-2+\y)--(\x+2,0+\y)--(\x+0,0+\y);
\draw [ fill=yellow!80] (\x+4,0+\y)--(\x+2,2+\y)--(\x+2,0+\y)--(\x+4,0+\y);
\draw [ fill=green!80] (\x+4,0+\y)--(\x+2,-2+\y)--(\x+2,0+\y)--(\x+4,0+\y);
}
\foreach \x in {0}
\foreach \y in {-8}
{
\draw [ fill=red!80] (\x+0,0+\y)--(\x+2,2+\y)--(\x+2,0+\y)--(\x+0,0+\y);
\draw [ fill=yellow!80] (\x+0,0+\y)--(\x+2,-2+\y)--(\x+2,0+\y)--(\x+0,0+\y);
\draw [ fill=yellow!80] (\x+4,0+\y)--(\x+2,2+\y)--(\x+2,0+\y)--(\x+4,0+\y);
\draw [ fill=green!80] (\x+4,0+\y)--(\x+2,-2+\y)--(\x+2,0+\y)--(\x+4,0+\y);
}

\end{tikzpicture}
\end{center}
\caption{The equivalent region of Wang tiles.}\label{fig_wang_region}
\end{figure}
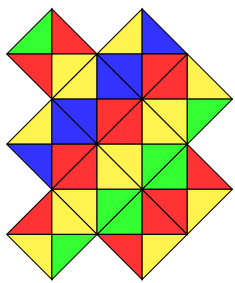

To prove the pattern in Figure \ref{fig_pattern} must be enforced, we establish the following facts step by step.

\begin{Fact}[Start with a selector]\label{fct_selector}
To tile the plane with the set of $29$ Wang bars we have constructed, the selectors must be used.  
\end{Fact}

\begin{proof}[Proof of Fact \ref{fct_selector}]
We prove this by cases.
\begin{itemize} 
    \item 
If the fillers of group $2$ are used, then the selectors must be used because the colors $a_8$ or $a_7$ in fillers of group $2$ can only be matched by the selectors. 
    \item 
If the linkers are used, then the selectors must be used. Because the colors $x$ or $y$ on the vertical sides of the linkers can only be matched by either the selectors directly or the fillers of group $2$.
    \item 
If the fillers of group $1$ are used, then the selectors must be used. Because the colors $b_3$ or $b_4$ in fillers of group $1$ can only be matched by the linkers.
    \item 
If the encoders are used, then the selectors must be used. Because the colors $a_3$ or $a_4$ in encoders can only be matched by the linkers or the fillers of group $1$.
    \item 
If the aligners are used, then the selectors must be used. Because the color $a_5$ or $a_6$ in aligners can only be matched by the fillers of group $1$ or $2$.
\end{itemize}
In all cases, the selectors must be used. \end{proof}

By Fact \ref{fct_selector}, we can start tiling by placing a selector on the plane first. The only other Wang bar that can be adjacent to the side colored $a_2$ of the top pointer of the selector is a locator of a section of an encoder. The same applies to the bottom pointer of the selector. So we have proved the following fact.

\begin{Fact}[Add encoders next to a selector] \label{fct_add_encoders}
    To tile the plane from any partial tiling, we must place an encoder next to the top pointer and bottom pointer of any selector.  
\end{Fact}

So we place one encoder right above a selector, and place another encoder right under it (see Figure \ref{fig_partial_1}). Note that only one section of each encoder is depicted in Figure \ref{fig_partial_1}, this means the selector picks one section out of the $m$ sections of each encoder.

\begin{figure}[H]
\begin{center}
\begin{tikzpicture}[scale=0.4]

\foreach \x in {0}
{
\draw [fill=blue!10] (\x+0,0)--(\x+0,1)--(\x+8,1)--(\x+8,4)--(\x+9,4)--(\x+9,1)--(\x+17,1)--(\x+17,0)--(\x+9,0)--(\x+9,-3)--(\x+8,-3)--(\x+8,0)--(\x+0,0);
\draw [fill=orange!10] (\x+17,5)--(\x+0,5)--(\x+0,4)--(\x+17,4)--(\x+17,5); \draw [fill=orange!60](\x+8,4)--(\x+8,5)--(\x+9,5)--(\x+9,4)--(\x+8,4);

}

\foreach \x in {0}
{
\draw [fill=orange!10] (\x+17,-3)--(\x+0,-3)--(\x+0,-4)--(\x+17,-4)--(\x+17,-3); \draw [fill=orange!60](\x+8,-4)--(\x+8,-3)--(\x+9,-3)--(\x+9,-4)--(\x+8,-4);
}

\draw [fill=red!60] (2,4)--(3,4)--(3,4.5)--(2,4.5)--(2,4); \draw [fill=red!60] (12,4)--(13,4)--(13,4.5)--(12,4.5)--(12,4);
\draw [fill=red!60] (4,-3)--(5,-3)--(5,-3.5)--(4,-3.5)--(4,-3); \draw [fill=red!60] (14,-3)--(15,-3)--(15,-3.5)--(14,-3.5)--(14,-3);

\end{tikzpicture}
\end{center}
\caption{A selector and two encoders.}\label{fig_partial_1}
\end{figure}
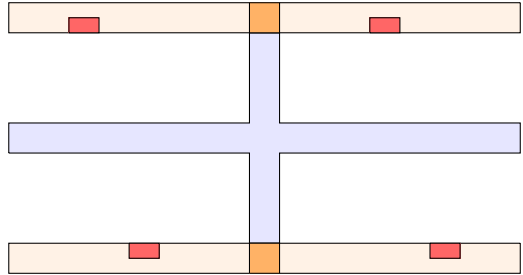

We highlight $4$ encoding colors by red rectangles in Figure \ref{fig_partial_1}. There is exactly one encoding color on one side of an arm, and we can continue tiling the plane from those encoding colors by the following fact.

\begin{Fact}[Add a linker next to an encoding color]\label{fct_partial_1}
    To tile the plane from the partial tiling illustrated in Figure \ref{fig_partial_1}, we must place a linker right next to each of the $4$ highlighted (red rectangles) encoding colors. More precisely, we must place a linker directly under the unit segment with encoding color $a_4$, or directly above encoding color $a_3$. 
\end{Fact}

\begin{figure}[H]
\begin{center}
\begin{tikzpicture}[scale=0.4]

\foreach \x in {0}
{
\draw [fill=blue!10] (\x+0,0)--(\x+0,1)--(\x+8,1)--(\x+8,4)--(\x+9,4)--(\x+9,1)--(\x+17,1)--(\x+17,0)--(\x+9,0)--(\x+9,-3)--(\x+8,-3)--(\x+8,0)--(\x+0,0);
\draw [fill=orange!10] (\x+17,5)--(\x+0,5)--(\x+0,4)--(\x+17,4)--(\x+17,5); \draw [fill=orange!60](\x+8,4)--(\x+8,5)--(\x+9,5)--(\x+9,4)--(\x+8,4);
\draw [fill=green!10] (\x+2,4)--(\x+3,4)--(\x+3,3)--(\x+2,3)--(\x+2,4);
\draw [fill=gray!10] (\x+5,1)--(\x+5,2)--(\x-1,2)--(\x-0.7,2.7)--(\x-1,3)--(\x+6,3)--(\x+6,1)--(\x+5,1);
}

\foreach \x in {20}
{
\draw [fill=blue!10] (\x+0,0)--(\x+0,1)--(\x+8,1)--(\x+8,4)--(\x+9,4)--(\x+9,1)--(\x+17,1)--(\x+17,0)--(\x+9,0)--(\x+9,-3)--(\x+8,-3)--(\x+8,0)--(\x+0,0);
\draw [fill=orange!10] (\x+17,5)--(\x+0,5)--(\x+0,4)--(\x+17,4)--(\x+17,5); \draw [fill=orange!60](\x+8,4)--(\x+8,5)--(\x+9,5)--(\x+9,4)--(\x+8,4);
\draw [fill=green!10] (\x+2,4)--(\x+3,4)--(\x+3,3)--(\x+2,3)--(\x+2,4);
\draw [fill=gray!10] (\x+6,4)--(\x+6,3)--(\x-1,3)--(\x-0.7,2.7)--(\x-1,2)--(\x+7,2)--(\x+7,4)--(\x+6,4);
}

\end{tikzpicture}
\end{center}
\caption{Putting a filler under the encoding color leads to a dead end.}\label{fig_1arm}
\end{figure}
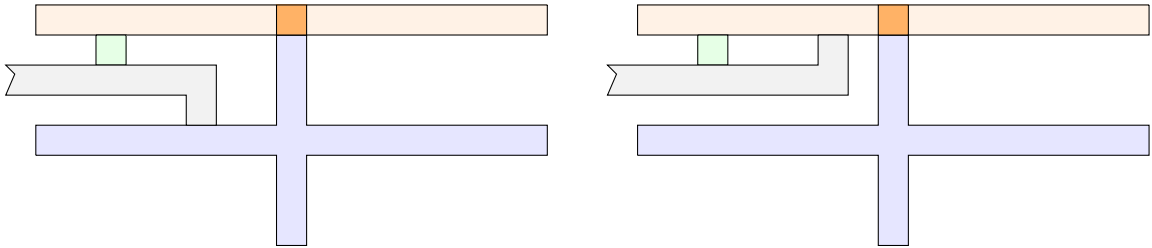

\begin{proof}[Proof of Fact \ref{fct_partial_1}] We prove by contradiction. If we do not place a linker right under the unit line segment with encoding color $a_4$, then the only other choice is to place a filler of group $1$ (see Figure \ref{fig_1arm}), which is colored $a_4$ on the top side and $b_3$ at the bottom side. Now the only group of Wang bars that can be placed right under the filler is a linker, which will create a contradiction with the color on the left arm of the selector (see the left of Figure \ref{fig_1arm}) because there is no encoding colors on top side of the left arm of the selector, or create a contradiction with the left arm of the picked-out section of encoder (see the right of Figure \ref{fig_1arm}) because there is not another encoding color $a_4$ there. This completes the proof of the fact.
\end{proof}

By Fact \ref{fct_partial_1}, we put a linker right under the unit segment colored $a_4$ of a section (of an encoder) which is picked out by the selector. On the other side of this linker, we must place another encoder because only the encoders have the color $a_3$ on the top sides (see Figure \ref{fig_partial_2}).

\begin{figure}[H]
\begin{center}
\begin{tikzpicture}[scale=0.4]

\draw [fill=blue!10] (0,0)--(0,1)--(8,1)--(8,4)--(9,4)--(9,1)--(17,1)--(17,0)--(9,0)--(9,-3)--(8,-3)--(8,0)--(0,0);

\draw [fill=orange!10] (17,5)--(0,5)--(0,4)--(17,4)--(17,5); \draw [fill=orange!60](8,4)--(8,5)--(9,5)--(9,4)--(8,4);
\draw [fill=orange!10] (-24,0)--(-7,0)--(-7,1)--(-24,1)--(-24,0); \draw [fill=orange!60](-15,0)--(-15,1)--(-16,1)--(-16,0)--(-15,0);

\draw [fill=gray!10] (3,4)--(3,3)--(-5,3)--(-4.7,2.7)--(-5,2)--(4,2)--(4,4)--(3,4);
\draw [fill=gray!10] (-12,1)--(-12,3)--(-5.3,3)--(-5,2.7)--(-5.3,2)--(-11,2)--(-11,1)--(-12,1);

\end{tikzpicture}
\end{center}
\caption{Connecting to the second encoder.}\label{fig_partial_2}
\end{figure}
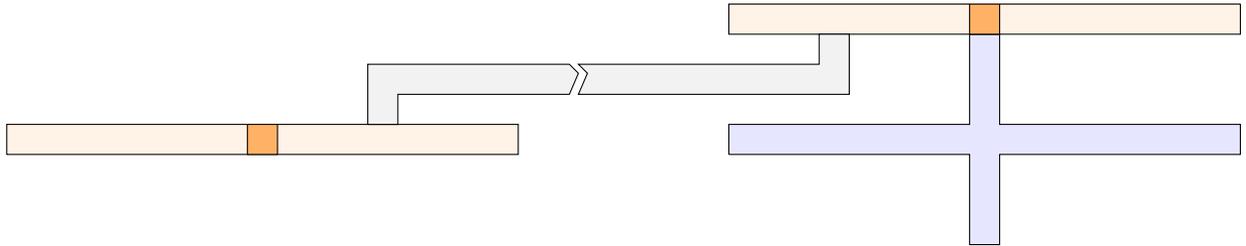

 We observe the following fact about how a linker connects two encoders.

\begin{Fact}[Add another encoder on the other side of a linker, weak version] \label{fct_align}
    To tile the plane from the partial tiling as illustrated in Figure \ref{fig_partial_2}, the linker must connect the left arm of a section of one encoder to the right arm of a section of another encoder. In other words, it can connect neither a left arm to a left arm, nor a right arm to a right arm.
\end{Fact}

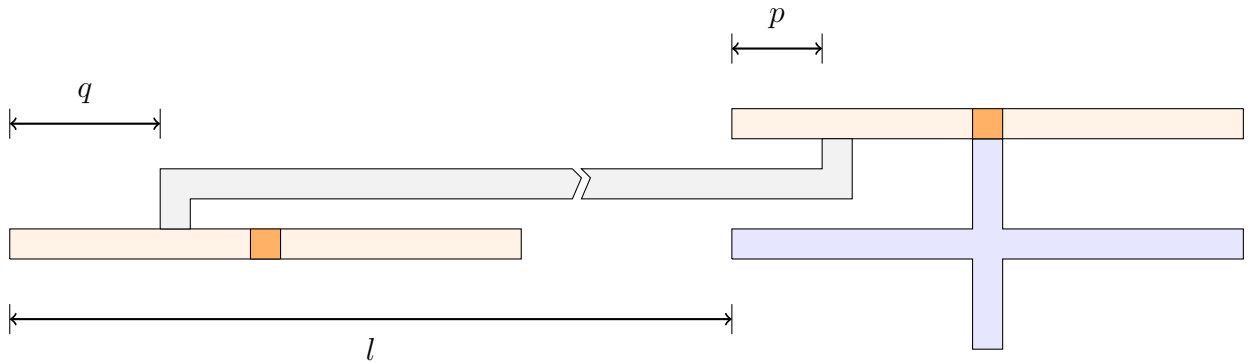
\begin{figure}[H]
\begin{center}
\begin{tikzpicture}[scale=0.4]

\draw [fill=blue!10] (0,0)--(0,1)--(8,1)--(8,4)--(9,4)--(9,1)--(17,1)--(17,0)--(9,0)--(9,-3)--(8,-3)--(8,0)--(0,0);

\draw [fill=orange!10] (17,5)--(0,5)--(0,4)--(17,4)--(17,5); \draw [fill=orange!60](8,4)--(8,5)--(9,5)--(9,4)--(8,4);

\draw [fill=orange!10] (-24,0)--(-7,0)--(-7,1)--(-24,1)--(-24,0); \draw [fill=orange!60](-15,0)--(-15,1)--(-16,1)--(-16,0)--(-15,0);

\draw [fill=gray!10] (3,4)--(3,3)--(-5,3)--(-4.7,2.7)--(-5,2)--(4,2)--(4,4)--(3,4);
\draw [fill=gray!10] (-19,1)--(-19,3)--(-5.3,3)--(-5,2.7)--(-5.3,2)--(-18,2)--(-18,1)--(-19,1);

\draw [<->,thick] (-24,-2)--(0,-2);  
\node at (-12,-3) {$l$};
\draw (-24,-1.5)--(-24,-2.5); 
\draw (0,-1.5)--(0,-2.5);

\draw [<->,thick] (0,7)--(3,7);  
\node at (1.5,8) {$p$};
\draw (3,6.5)--(3,7.5); 
\draw (0,6.5)--(0,7.5);

\draw [<->,thick] (-19,4.5)--(-24,4.5);  
\node at (-21.5,5.5) {$q$};
\draw (-19,4)--(-19,5); 
\draw (-24,4)--(-24,5);

\end{tikzpicture}
\end{center}
\caption{Connecting two left arms of two encoders by a linker leads to a dead end.}\label{fig_2arm}
\end{figure}

\begin{proof}[Proof of Fact \ref{fct_align}]
    We prove the fact by contradiction for the general case where our set of Wang bars is simulating a set of $n$ Wang tiles with $m$ different colors. Suppose the linker is connecting two left arms as illustrated in Figure \ref{fig_2arm}. Note that only one section of each of the two encoders is depicted (and the locators for the sections are in darker orange). Suppose the linker is connecting the $(p+1)$-th unit square of the upper section to the $(q+1)$-th unit square of the lower section. Both $p$ and $q$ satisfy $0\leq p,q\leq m-1$. Hence
    $$-(m-1)\leq q-p \leq m-1.$$
    Then the distance between the left side of the lower section and the left side of the selector is 
    \begin{align*} 
l &=  \Big ( (n-1)(2m+1)+m+1 \Big ) -1-p+q \\ 
  &=  (n-1)(2m+1) +(m+q-p).
\end{align*}
Therefore, 
$$(n-1)(2m+1) +1 \leq  l \leq (n-1)(2m+1)+2m-1,$$
which implies $l$ is not a multiple of $2m+1$. On the other hand, the only Wang bars that can lie between the lower section (of an encoder) and the selector in Figure \ref{fig_2arm} are another section or an aligner, both of which are of length $2m+1$. So $l$ is a multiple of $2m+1$, which leads to a contradiction.
\end{proof}

By applying the same argument as in the proof of Fact \ref{fct_align}, we can obtain the following stronger fact.

\begin{Fact}[Add another encoder on the other side of a linker, strong version] \label{fct_align_2}
    To tile the plane from the partial tiling as illustrated in Figure \ref{fig_partial_2}, the linker must connect the left arm of a section of one encoder to the right arm of a section of another encoder. Furthermore, if the linker is connecting the $p$-th unit square of the left arm to the $q$-th unit square of the right arm, then $p=q$.
\end{Fact}

We continue tiling the plane from the partial tiling in Figure \ref{fig_partial_2}. Note that in Figure \ref{fig_partial_2}, the linker connects the fourth unit square of a left arm and the fourth unit square of a right arm, which agrees with Fact \ref{fct_align} and Fact \ref{fct_align_2}. We continue by placing the next group of Wang bars on top of the locator of the second section, and we have the following fact.

\begin{Fact} [Add another selector] \label{fct_add_2nd_selector}
    To tile the plane from the partial tiling as illustrated in Figure \ref{fig_partial_2}, the next group of Wang bars to be placed on top of the locator marked by dark orange (of the second encoder) must be a selector.
\end{Fact}

\begin{figure}[H]
\begin{center}
\begin{tikzpicture}[scale=0.4]

\draw [fill=blue!10] (0,0)--(0,1)--(8,1)--(8,4)--(9,4)--(9,1)--(17,1)--(17,0)--(9,0)--(9,-3)--(8,-3)--(8,0)--(0,0);

\draw [fill=orange!10] (17,5)--(0,5)--(0,4)--(17,4)--(17,5); \draw [fill=orange!60](8,4)--(8,5)--(9,5)--(9,4)--(8,4);
\draw [fill=orange!10] (-24,0)--(-7,0)--(-7,1)--(-24,1)--(-24,0); \draw [fill=orange!60](-15,0)--(-15,1)--(-16,1)--(-16,0)--(-15,0);

\draw [fill=gray!10] (3,4)--(3,3)--(-5,3)--(-4.7,2.7)--(-5,2)--(4,2)--(4,4)--(3,4);
\draw [fill=gray!10] (-12,1)--(-12,3)--(-5.3,3)--(-5,2.7)--(-5.3,2)--(-11,2)--(-11,1)--(-12,1);

\draw [fill=green!10](-15,2)--(-15,1)--(-16,1)--(-16,2)--(-15,2);

\draw [<->, very thick,red] (-13.5,2.5)--(-12.5,2.5);



\foreach \x in {-18}
{
\draw [fill=gray!10] (\x-6,2)--(\x+4,2)--(\x+4,4)--(\x+3,4)--(\x+3,3)--(\x-6,3);
}

\end{tikzpicture}
\end{center}
\caption{Putting a filler on top of the second locator leads to a dead end.}\label{fig_next}
\end{figure}
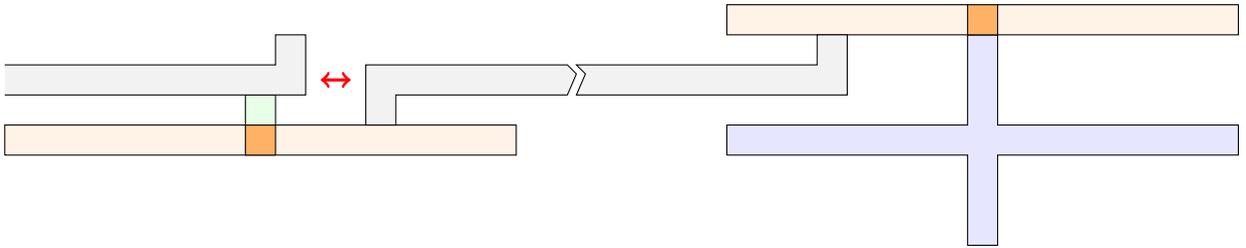

\begin{proof}[Proof of Fact \ref{fct_add_2nd_selector}]
We prove this by contradiction. Suppose we do not place another selector on top of the second locator, then it is easy to check the only other choice is to put a filler (see Figure \ref{fig_next}). This filler must have color $a_1$ on its bottom side and color $b_4$ on the top side. By looking for a group of Wang bars with color $b_4$ at the bottom, the only group to be placed on top of the filler is a linker. Recall that the left side of the middle Wang bar of the linker is colored $y$, and the right side is colored $x$. Hence no Wang bars can be put between the two linkers to connect the colors $x$ to $y$ (see the red double arrow in Figure \ref{fig_next}).
\end{proof}

\begin{figure}[H]
\begin{center}
\begin{tikzpicture}[scale=0.4]

\draw [fill=blue!10] (0,0)--(0,1)--(8,1)--(8,4)--(9,4)--(9,1)--(17,1)--(17,0)--(9,0)--(9,-3)--(8,-3)--(8,0)--(0,0);

\draw [fill=orange!10] (17,5)--(0,5)--(0,4)--(17,4)--(17,5); \draw [fill=orange!60](8,4)--(8,5)--(9,5)--(9,4)--(8,4);
\draw [fill=orange!10] (-24,0)--(-7,0)--(-7,1)--(-24,1)--(-24,0); \draw [fill=orange!60](-15,0)--(-15,1)--(-16,1)--(-16,0)--(-15,0);

\draw [fill=gray!10] (3,4)--(3,3)--(-5,3)--(-4.7,2.7)--(-5,2)--(4,2)--(4,4)--(3,4);
\draw [fill=gray!10] (-12,1)--(-12,3)--(-5.3,3)--(-5,2.7)--(-5.3,2)--(-11,2)--(-11,1)--(-12,1);

\foreach \x in {-24}
\foreach \y in {4}
{
\draw [fill=blue!10] (\x+0,0+\y)--(\x+0,1+\y)--(\x+8,1+\y)--(\x+8,4+\y)--(\x+9,4+\y)--(\x+9,1+\y)--(\x+17,1+\y)--(\x+17,0+\y)--(\x+9,0+\y)--(\x+9,-3+\y)--(\x+8,-3+\y)--(\x+8,0+\y)--(\x+0,0+\y);
}

\end{tikzpicture}
\end{center}
\caption{Adding another selector.}\label{fig_done}
\end{figure}
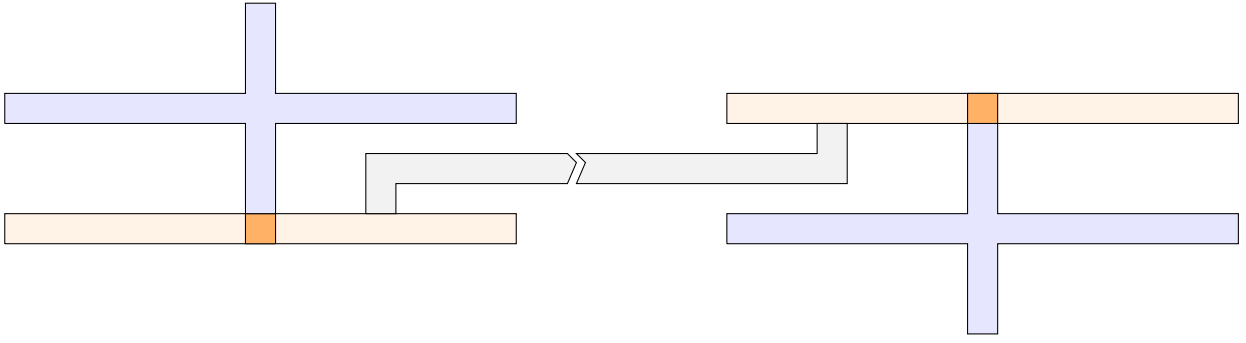

So we must place a second selector as illustrated in Figure \ref{fig_done}. From this partial tiling, we can then repeatedly add an encoder, a linker, or a selector, by applying Fact \ref{fct_add_encoders} to Fact \ref{fct_add_2nd_selector}. By Fact \ref{fct_align} and Fact \ref{fct_align_2}, the horizontal gaps between encoders and the arms of selectors can be filled by aligners. This forms the backbone of the tiling. It can be checked that the rest area can be filled with fillers of group $1$ or group $2$. The resulting tiling is exactly in the pattern illustrated in Figure \ref{fig_pattern}.

Finally, by Fact \ref{fct_align_2}, the linkers always connect the sides of the arms of the encoders which encode the same color of the Wang tiles. So the set of $29$ Wang bars can tile the entire plane if and only if the corresponding set of Wang tiles can. This completes the proof of Theorem \ref{thm_main}. \end{proof}

\subsection{Consequence on Color Deficiency}

\begin{Corollary}
    Tiling the plane with a set of Wang tiles with color deficiency of $25$ is undecidable.
\end{Corollary}

\begin{proof}
    We prove by the same method as the proof of Theorem \ref{thm_cd} in \cite{jr12}, by reduction from the tiling problem of a set of $29$ Wang bars. Given a set $B$ of $29$ Wang bars in the forms in the proof of Theorem \ref{thm_main}, we construct a set $T$ of Wang tiles in the following way. If there is a Wang bar with a length of at least $l\geq 2$, we cut it into two pieces of length $1$ and $l-1$, and assign a unique new color to the two new vertical sides of the two new Wang bars. Repeat this operation, each time introducing a new color to both of the vertical sides of the set, until all the Wang bars are of length $1$, which forms a set of Wang tiles. Assume we cut $t$ times in total. Hence the resulting set $T$ of Wang tiles has a total number of $n(T)=29+t$ tiles. 
    
    Note that the number of different colors on the east or west sides of $B$ is $4$ (colors $x$, $y$, $z$, and $0$), 
    hence the number of different colors on the east or west sides of the corresponding set of Wang tiles is $4+t$. This means $c(T)\geq 4+t$. Therefore,
    $$ cd(T)=n(T)-c(T)\leq (29+t)-(4+t)=25.$$
    Note that if $cd(T)<25$, it is easy to add tiles (to $T$) which do not contribute to tiling the plane but increase the color deficiency. So we construct a set $T$ of Wang tiles $T$ with $cd(T)=25$ which tiles the plane if and only if the set $B$ of Wang bars tiles the plane. This completes the proof.
\end{proof}

\section{Conclusion}\label{sec_con}

By introducing novel techniques to Ollinger's framework, we improve the undecidability results of Jeandel and Rolin on the tiling problems with a set of a fixed number of Wang bars, and the tiling problem with a set of Wang tiles with fixed color deficiency. It is interesting to investigate whether these tiling problems remain undecidable for smaller fixed parameters. 

\section*{Acknowledgements}
The first author was supported by the Research Fund of Guangdong University of Foreign Studies (Nos. 297-ZW200011 and 297-ZW230018), and the National Natural Science Foundation of China (No. 61976104).



\begin{thebibliography}{99}

\bibitem{b66} R. Berger, The undecidability of the domino problem, \textit{Memoirs of the American Mathematical Society}, \textbf{66}(1966), 1-72.

\bibitem{hln23}
B. Hellouin de Menibus, V.H. Lutfalla, C. No\^us, The domino problem is undecidable on every rhombus subshift, In: F. Drewes, M. Volkov, (eds), Developments in Language Theory (DLT 2023), Lecture Notes in Computer Science, vol 13911. Springer, Cham. (2023), 100-112.

\bibitem{j10}
E. Jeandel, The periodic domino problem revisited, \textit{Theoretical Computer Science}, \textbf{411}(44–46) (2010), 4010-4016.

\bibitem{jr12}
E. Jeandel, N. Rolin, Fixed parameter undecidability for Wang tilesets, In: E. Formenti (eds), AUTOMATA and JAC 2012 conferences, EPTCS 90, (2012), 69–85.

\bibitem{k94}
J. Kari, Reversibility and surjectivity problems of cellular automata, \textit{Journal of Computer and System Sciences}, \textbf{48(1)}(1994), 149-182.

\bibitem{k03}
J. Kari, Infinite snake tiling problems. In: M. Ito, M. Toyama, (eds), Developments in Language Theory (DLT 2002), Lecture Notes in Computer Science, vol 2450. Springer, Berlin, Heidelberg, (2003), 67-77.

\bibitem{k08}
J. Kari, On the undecidability of the tiling problem. In:  V. Geffert, J. Karhum\"aki, A. Bertoni, B. Preneel, P. N\'avrat, M. Bielikov\'a, (eds), Theory and Practice of Computer Science (SOFSEM 2008), Lecture Notes in Computer Science, vol 4910. Springer, Berlin, Heidelberg, (2008), 74-82.


\bibitem{m08} 
M. Margenstern, The domino problem of the hyperbolic plane is undecidable, \textit{Theoretical Computer Science}, \textbf{407(1–3)}(2008), 29-84.

\bibitem{m23}
M. Margenstern, The domino problem of the hyperbolic plane is undecidable: new proof, \textit{Complex Systems}, \textbf{32(1)}(2023), 19–56.

\bibitem{o08}
N. Ollinger, Two-by-two substitution systems and the undecidability of the domino problem. In: A. Beckmann, C. Dimitracopoulos, B. L\"owe, (eds), Logic and Theory of Algorithms (CiE 2008). Lecture Notes in Computer Science, vol 5028. Springer, Berlin, Heidelberg, 476-485

\bibitem{o09}
N. Ollinger, Tiling the plane with a fixed number of polyominoes, In: A.H. Dediu, A.M. Ionescu, C. Martín-Vide (eds), Language and Automata Theory and Applications (LATA 2009). Lecture Notes in Computer Science, vol 5457. Springer, Berlin, Heidelberg, 638-649.

\bibitem{p14}
B. Poonen, Undecidable problems: a sampler, In: J. Kennedy (eds), Interpreting G\"odel: Critical Essays, Cambridge University Press, (2014), 211-241.


\bibitem{r71}
R. M. Robinson, Undecidability and nonperiodicity for tilings of the plane, \textit{Inventiones Mathematicae}, \textbf{12}(1971), 177–209. 

\bibitem{r78}
R. M. Robinson, Undecidable tiling problems in the hyperbolic plane, \textit{Inventiones Mathematicae}, \textbf{44}(1978), 259–264. 

\bibitem{wang61}
H. Wang, Proving theorems by pattern recognition- II, \textit{Bell System Technical Journal}, \textbf{40}(1961) 1-41.

\bibitem{yang23} 
C. Yang, Tiling the plane with a set of ten polyominoes, \textit{International Journal of Computational Geometry \& Applications}, \textbf{33}(03n04)(2023), 55-64.


\bibitem{yang24} 
C. Yang, On the undecidability of tiling the plane with a set of 9 polyominoes (in Chinese), (2024), to appear in \textit{SCIENTIA SINICA Mathematica}.

\bibitem{yz24} 
C. Yang, Z. Zhang, A proof of Ollinger's conjecture: undecidability of tiling the plane with a set of 8 polyominoes, arXiv:2403.13472 [math.CO]

\bibitem{y14}
J. Yang, Rectangular tileability and complementary tileability are undecidable, \textit{European Journal of Combinatorics}, \textbf{41}(2014), 20-34.


\end{thebibliography}
\end{document}